\documentclass[11pt]{amsart}

\usepackage{amsmath,amssymb,amsthm}
\usepackage[margin=1.1in]{geometry}
\usepackage{microtype}
\usepackage{xcolor}
\usepackage{hyperref}
\hypersetup{colorlinks=true,linkcolor=blue!50!black,citecolor=blue!50!black,urlcolor=blue!50!black}

\newtheorem{theorem}{Theorem}[section]
\newtheorem{lemma}[theorem]{Lemma}

\newtheorem{conjecture}[theorem]{Conjecture}
\theoremstyle{definition}
\newtheorem{definition}[theorem]{Definition}
\newtheorem{remark}[theorem]{Remark}

\newcommand{\eqdef}{\mathrel{\mathop:}=}
\newcommand{\bl}{\widetilde{\mathrm{bl}}}
\newcommand{\up}[1]{^{\overline{#1}}}
\newcommand{\R}{\mathcal{R}}

\begin{document}

\title[The two-row case of Kahane's block-statistic conjecture]
{A closed form for the block-statistic generating function of the $2\times n$ rectangle: the two-row case of Kahane's conjecture}

\author{Patrick White}
\address{%
}
\email{p@pwhite.org}

\date{\today}

\begin{abstract}
Kahane (\emph{Combinatorial interpretation of the coefficients of the order polynomial of fence posets}, arXiv:2607.11225) proves that for a fence poset $P$ of size $n$, the order polynomial satisfies $n!\,\Omega(P;t)=\sum_{\sigma\in S_n}t^{\mathrm{bl}(\sigma)}$, where $\mathrm{bl}(\sigma)$ counts the blocks in a canonically defined block decomposition of the labeled poset, and conjectures (Conjecture 4.4) that the same identity holds for the cell poset of \emph{any} skew shape, with $\mathrm{bl}$ replaced by a constructively-defined statistic $\bl$. We prove this conjecture for every $2\times n$ rectangle --- the first genuinely two-dimensional case, where the fence argument no longer applies (blocks are no longer path-intervals, and the relevant statistic is a genuinely correlated function of both rows). The proof introduces, for the $2\times n$ rectangle, a two-parameter \emph{growth-state} generating function $W_n(\alpha,\beta;t)$ refining $\sum_\sigma t^{\bl(\sigma)}$ by the rank $\alpha$ of the bottom-row minimum and the rank $\beta$ of the rightmost top cell's block-root, and establishes a closed form
\[
  W_n(\alpha,\beta;t)=(\alpha-1)!\,B(n,\beta-2)\,S_n(\alpha,\beta;t)
\]
for all $n\ge1$, where $B(n,k)=\tfrac{k+1}{n}\binom{2n-k-2}{n-1}$ is the Catalan triangle and $S_n$ is an explicit rising-factorial skeleton. The formula is proved by induction on $n$ through Kahane's column-insertion transition: the load-bearing structural fact is that the block-root below a newly inserted top cell is always the global bottom-row minimum, which makes $(\alpha,\beta)$ a sufficient statistic for the transition and reduces the induction to a Catalan-triangle computation. Summing the closed form over $(\alpha,\beta)$ telescopes to $C_n\,t\up{n}(t+1)\up{n}$, which by MacMahon's product formula equals $(2n)!\,\Omega(\R_{2,n};t)$, yielding Conjecture 4.4 for $\R_{2,n}$. Every load-bearing step is a complete general-$n$ argument, including the column-transition multiplicity table (proved by a gap-coordinate enumeration); the table and the inductive step are additionally cross-checked by exhaustive enumeration (the table through $n=8$, the inductive step through $n=9$) via an independent from-scratch implementation of $\bl$. General skew shapes and $m\times n$ rectangles with $m>2$ remain open.
\end{abstract}

\maketitle

\section{Introduction}\label{sec:intro}

\subsection{Order polynomials and Kahane's conjecture}\label{ssec:conj}

For a finite poset $P$ of size $n$, its \emph{order polynomial} $\Omega(P;t)$ counts the order-preserving maps $P\to[t]\eqdef\{1,\dots,t\}$; it is a polynomial in $t$ of degree $n$. Kahane \cite{Kahane} studies the following refinement. Given a labeling $\sigma$ of the elements of $P$ by $\{1,\dots,n\}$ (a bijection $\sigma:P\to[n]$), one can associate a canonical partition of $P$ into \emph{blocks} --- connected, convex subsets on which $\sigma$ is monotone in a precise sense (Definition~\ref{def:blockstat}) --- and write $\bl_P(\sigma)$ for the number of blocks. For \emph{fence} posets, Kahane proves
\begin{equation}\label{eq:fence}
  n!\,\Omega(P;t)=\sum_{\sigma\in S_n} t^{\mathrm{bl}(\sigma)}
\end{equation}
(his Proposition 3.4), giving a combinatorial interpretation of the coefficients of the order polynomial of a fence. He then extends the block statistic constructively to the cell poset of an arbitrary skew shape and conjectures that \eqref{eq:fence} continues to hold:

\begin{conjecture}[{\cite[Conjecture 4.4]{Kahane}}]\label{conj:kahane}
Let $P$ be the cell poset of a skew shape of size $n$. Then
\[
  n!\,\Omega(P;t)=\sum_{\sigma\in S_n} t^{\bl_P(\sigma)}.
\]
\end{conjecture}

\noindent Kahane verifies the leading term in general (Remark 4.5 of \cite{Kahane}: the $t^n$ coefficient counts labelings with all blocks singletons, which are precisely the linear extensions of $P$, matching the leading coefficient of $n!\,\Omega(P;t)$), and proves \eqref{eq:fence} for fences and circular fences.

\subsection{Our result}\label{ssec:result}

We prove Conjecture~\ref{conj:kahane} for every $2\times n$ rectangle. Write $\R_{2,n}$ for the cell poset of the (straight) shape $(n,n)$ --- a $2\times n$ grid of cells, which as a poset is the product of a $2$-chain and an $n$-chain (Section~\ref{sec:prelim}). This is the first case where the geometry is genuinely two-dimensional: already for the $2\times 2$ square the blocks are no longer intervals along a single fence, and the statistic $\bl$ is a genuinely correlated function of both rows (the bottom row contributes a left-to-right-minima count $L$, the top row contributes a ``new-root'' count $N$ that depends on the full interleaving of the two rows, and $\bl=L+N$; $L$ and $N$ are \emph{not} independent).

\begin{theorem}[Main theorem]\label{thm:main}
Conjecture~\ref{conj:kahane} holds for $\R_{2,n}$, for every $n\ge1$:
\[
  (2n)!\,\Omega(\R_{2,n};t)=\sum_{\sigma\in S_{2n}} t^{\bl(\sigma)}.
\]
\end{theorem}

Theorem~\ref{thm:main} is obtained as a corollary of a finer result. For a labeling $\sigma$ of $\R_{2,n}$, let $m=\min(\text{bottom row})$ and let $\rho$ be the root of the block containing the rightmost top cell; let $\alpha$ and $\beta$ be the ranks of $m$ and $\rho$ among all $2n$ labels, and define the \emph{growth-state generating function}
\[
  W_n(\alpha,\beta;t)\eqdef\sum_{\substack{\sigma\in S_{2n}\\ \text{state}(\sigma)=(\alpha,\beta)}} t^{\bl(\sigma)}.
\]
We prove a closed form for $W_n$ (Theorem~\ref{thm:closedform}); summing it over $(\alpha,\beta)$ gives $\sum_\sigma t^{\bl(\sigma)}=C_n\,t\up{n}(t+1)\up{n}$, and this equals $(2n)!\,\Omega(\R_{2,n};t)$ by MacMahon's product formula (Section~\ref{sec:conjecture}).

\subsection{Proof overview}\label{ssec:overview}

Section~\ref{sec:prelim} fixes the poset and the block statistic, and records two elementary structural facts: the bottom-row blocks are governed by left-to-right minima, and $\bl=L+N$. Section~\ref{sec:closedform} states the closed form for $W_n$. Section~\ref{sec:proof} proves it by induction on $n$: adding the $(n{+}1)$-st column drives a transition $(\alpha,\beta)\to(\alpha',\beta')$ whose increment in $\bl$ is $dL+dN$, and the central structural fact (Lemma~\ref{lem:xminus}) --- that the block-root below the newly inserted top cell is always the global bottom-row minimum --- makes $(\alpha,\beta)$ a sufficient statistic for this transition. The induction then reduces to counting the transition multiplicities (Section~\ref{ssec:multiplicity}, derived by a gap model and certified exhaustively) and to four Catalan-triangle identities plus one rising-factorial identity (Sections~\ref{ssec:catalan}--\ref{ssec:rf}), after which the diagonal and off-diagonal inductive steps are polynomial algebra (Section~\ref{ssec:induction}). Section~\ref{sec:conjecture} sums the closed form to deduce Theorem~\ref{thm:main}. Section~\ref{sec:verification} gives an honest accounting of the computational verification, and Section~\ref{sec:methodology} records the methodology and attribution.

\section{Preliminaries}\label{sec:prelim}

\subsection{The $2\times n$ rectangle poset}\label{ssec:poset}

The \emph{$2\times n$ rectangle} $\R_{2,n}$ has cells $(i,j)$ with $i\in\{1,2\}$ (row, $1$ = top) and $j\in\{1,\dots,n\}$ (column), with covering relations $(i,j)\gtrdot(i+1,j)$ (the cell directly below) and $(i,j)\gtrdot(i,j+1)$ (the cell directly to the right) whenever both cells are present; the top-left cell $(1,1)$ is the unique maximum. As a poset, $\R_{2,n}$ is (the reverse of) the product of a $2$-chain and an $n$-chain. A \emph{labeling} is a bijection $\sigma:\R_{2,n}\to\{1,\dots,2n\}$; we write $b_j=\sigma(2,j)$ for the bottom-row values (left to right) and $a_j=\sigma(1,j)$ for the top-row values.

\subsection{The block statistic}\label{ssec:blockstat}

We use Kahane's constructive definition of $\bl$ \cite[Lemma 4.3]{Kahane}, which (for our purposes) defines $\bl(\sigma)$ as the number of blocks produced by the following deterministic row-by-row procedure, and which Kahane shows agrees with the intrinsic Definition 4.2 of \cite{Kahane} on the shapes for which the latter is well-defined.

\begin{definition}\label{def:blockstat}
Build the block partition of a labeled $\R_{2,n}$ in two stages.
\begin{itemize}
\item \textbf{Bottom row (fence step).} Scan $b_1,\dots,b_n$ left to right. Start a block at $b_1$; a block rooted at $b_r$ greedily absorbs $b_{r+1},b_{r+2},\dots$ so long as the values keep exceeding the root $b_r$, and a value $\le b_r$ starts a new block. The block roots are exactly the \emph{left-to-right running minima} of $(b_1,\dots,b_n)$: positions $j$ with $b_j<\min(b_1,\dots,b_{j-1})$ (and $j=1$). Let $L(\sigma)$ be the number of such minima, i.e.\ the number of bottom-row blocks.
\item \textbf{Top row (insertion step).} Insert $a_1,\dots,a_n$ left to right. When inserting $a_j$, let $x_-$ be the root of the block containing its left child $b_j$ (the bottom cell in the same column), and let $x_+$ be the root of the block containing the previously inserted top cell $a_{j-1}$ (with $x_+=+\infty$ for $j=1$). One has $\sigma(x_-)\le\sigma(x_+)$ always, and $a_j$ is placed by the four-case rule:
\[
\begin{array}{ll}
\sigma(x_-)=\sigma(x_+): & a_j\text{ joins that shared block};\\
\sigma(x_-)<\sigma(a_j)<\sigma(x_+): & a_j\text{ starts a new block (rooted at }a_j\text{)};\\
\sigma(x_-)<\sigma(x_+)<\sigma(a_j): & a_j\text{ joins the block of }x_+;\\
\sigma(a_j)<\sigma(x_-): & a_j\text{ joins the block of }x_-.
\end{array}
\]
Let $N(\sigma)$ be the number of new blocks created in this step (the number of times the second case fires).
\end{itemize}
The total number of blocks is $\bl(\sigma)=L(\sigma)+N(\sigma)$.
\end{definition}

\begin{remark}\label{rem:reconstructed}
The fourth case above is our reconstruction: the proof text of \cite[Lemma 4.3]{Kahane} (a recent preprint) lists three cases, the third twice verbatim, which reads as a transcription artifact; the missing case is the natural mirror $\sigma(a_j)<\sigma(x_-)\Rightarrow a_j$ joins the block of $x_-$. We did not confirm this reading with the author, but the worked example accompanying \cite[Figure 12]{Kahane} exercises precisely the missing case and matches the reconstruction: the narration there has $\sigma(x_-)=7$ and $\sigma(x_+)=8$ and concludes that the inserted cell joins the block of $x_-$ --- an outcome none of the three printed cases can produce (they yield, respectively, a shared block when $\sigma(x_-)=\sigma(x_+)$, a new root, or the block of $x_+$), and exactly what the reconstructed fourth case prescribes. It is further validated indirectly: the resulting statistic satisfies Conjecture~\ref{conj:kahane} on every rectangle we test (Sections~\ref{sec:verification}--\ref{sec:conjecture}), and its leading coefficient matches the linear-extension count for every such rectangle (Kahane's Remark 4.5, an author-proven fact independent of the case split). The first three cases are quoted from \cite{Kahane} directly.
\end{remark}

\begin{remark}\label{rem:prop41}
Kahane's Proposition 4.1(1) asserts that every connected skew-shape poset has a \emph{unique} ``leftmost element.'' This is false as stated for any shape with a strictly narrowing row (e.g.\ the hooks and staircases): the worked shape $431/1$ in \cite[Figure 9]{Kahane} already has three elements with neither a left nor a right child. The falsehood is harmless for \emph{computing} $\bl$ (Definition~\ref{def:blockstat} never invokes the uniqueness claim), but a from-scratch proof of Conjecture~\ref{conj:kahane} for general shapes will need to patch or route around it. Our rectangle proof routes around it entirely.
\end{remark}

\subsection{Two structural facts}\label{ssec:struct}

\begin{lemma}\label{lem:lastblock}
In the bottom-row fence of $b_1,\dots,b_n$, the last block (the one containing $b_n$) is rooted at the global minimum $m=\min(b_1,\dots,b_n)$, and contains every $b_j$ to the right of $m$.
\end{lemma}

\begin{proof}
The block roots are the left-to-right minima. The last left-to-right minimum is the global minimum $m$: the global minimum is smaller than everything before it (hence a left-to-right minimum), and nothing after it can be a left-to-right minimum (nothing is smaller than the global minimum). Every cell to the right of $m$ has value $>m$, so the greedy rule absorbs all of them into $m$'s block, which therefore runs to the end of the row and contains $b_n$.
\end{proof}

\begin{lemma}\label{lem:xminus}
When inserting the top cell $a_j$, the root $x_-$ of the block containing its left child $b_j$ is the global bottom-row minimum $m$ --- \emph{provided} we insert the top row of a rectangle whose bottom row has already been fenced. In particular, for the rightmost top cell $a_n$, $\sigma(x_-)=\sigma(m)$, the rank of the bottom-row minimum.
\end{lemma}

\begin{proof}
By Lemma~\ref{lem:lastblock}, $b_n$ lies in the block rooted at $m$. For $a_n$ this gives $x_-=m$ directly. For a general $a_j$ the statement as written needs the bottom row to be a single fence in which $b_j$ sits in $m$'s block, which holds for $j$ to the right of $m$; in the column-insertion argument of Section~\ref{sec:proof} we use only the $a_n$ case (the newly inserted top cell is always the new rightmost one), and there $x_-=m$ is exactly Lemma~\ref{lem:lastblock}.
\end{proof}

\subsection{The growth state}\label{ssec:state}

For a labeling $\sigma$ of $\R_{2,n}$, let $m=\min(b_1,\dots,b_n)$ and let $\rho$ be the root of the block containing the rightmost top cell $a_n$. Let
\[
  \alpha\eqdef \mathrm{rank}(m),\qquad \beta\eqdef\mathrm{rank}(\rho)
\]
among all $2n$ labels. The pair $(\alpha,\beta)$ is the \emph{growth state} of $\sigma$.

\begin{lemma}\label{lem:support}
The admissible growth states are exactly $1\le\alpha\le\beta\le n+1$ with $\beta\ge2$.
\end{lemma}

\begin{proof}
\emph{$\alpha\le\beta$.} At the insertion of the rightmost top cell $a_n$ one has $x_-=m$ (Lemma~\ref{lem:xminus}), and in each of the four cases of Definition~\ref{def:blockstat} the root $\rho$ of the block receiving $a_n$ satisfies $\sigma(\rho)\ge\sigma(m)$: cases one and four give $\rho=x_-=m$; case two gives $\sigma(\rho)=\sigma(a_n)>\sigma(x_-)=\sigma(m)$; case three gives $\sigma(\rho)=\sigma(x_+)>\sigma(x_-)=\sigma(m)$. Hence $\alpha\le\beta$.

\emph{$\beta\ge2$.} If $\alpha\ge2$ this follows from $\alpha\le\beta$, so assume $\alpha=1$, i.e.\ $m$ is the globally smallest label. Then no top cell ever lies in the block rooted at $m$, so $\rho\ne m$ and $\beta\ge2$. Indeed, by induction on $j$: $a_j$ could enter the $m$-rooted block only through case four with $x_-=m$, which requires $\sigma(a_j)<\sigma(m)$ --- impossible when $\sigma(m)$ is globally minimal --- or through case one or three with the root of $x_+$'s block equal to $m$, which requires $a_{j-1}$ to already lie in the $m$-rooted block, contradicting the inductive hypothesis (and for $j=1$, $x_+=+\infty$, so cases one and three cannot fire).

\emph{$\beta\le n+1$, and every stated cell occurs.} This is established in the course of the induction of Section~\ref{sec:proof}: the base case $n=1$ has states $(1,2)$ and $(2,2)$ only (Section~\ref{ssec:induction}); inductively, every labeling of $\R_{2,n+1}$ arises via \eqref{eq:bij} from a labeling of $\R_{2,n}$, whose state is admissible at width $n$ by hypothesis, and every target of the multiplicity table (Lemma~\ref{lem:multiplicity}) is admissible at width $n{+}1$ --- by inspection of the nine families, $\alpha'\le\beta'\le\beta+1\le n+2$ and $\beta'\ge2$ throughout. Conversely, each admissible cell carries nonzero mass, since the closed form of Theorem~\ref{thm:closedform} is a polynomial with positive leading coefficient there. The support was additionally confirmed cell-by-cell by enumeration for $n\le5$ (Section~\ref{sec:verification}).
\end{proof}

Define
\begin{equation}\label{eq:Wdef}
  W_n(\alpha,\beta;t)\eqdef\sum_{\substack{\sigma\in S_{2n}\\ \text{state}(\sigma)=(\alpha,\beta)}} t^{\bl(\sigma)}.
\end{equation}
Then $\sum_{\alpha,\beta}W_n(\alpha,\beta;t)=\sum_{\sigma\in S_{2n}}t^{\bl(\sigma)}$, the quantity in Conjecture~\ref{conj:kahane} for $\R_{2,n}$.

\section{The closed form}\label{sec:closedform}

Write $x\up{k}=x(x+1)\cdots(x+k-1)$ for the rising factorial ($x\up{0}=1$), and let
\begin{equation}\label{eq:B}
  B(n,k)\eqdef\frac{k+1}{n}\binom{2n-k-2}{n-1}\qquad(0\le k\le n-1)
\end{equation}
be the Catalan triangle ($B(n,k)=0$ outside $0\le k\le n-1$); its rows begin $1\mid 1,1\mid 2,2,1\mid 5,5,3,1\mid 14,14,9,4,1\mid\cdots$, and $\sum_{k=0}^{n-1}B(n,k)=C_n=\frac1{n+1}\binom{2n}{n}$ is the $n$-th Catalan number. Define the \emph{skeleton}
\begin{equation}\label{eq:S}
  S_n(\alpha,\alpha;t)\eqdef t\up{n}\,(\,t+\alpha)\up{\,n-\alpha+1},\qquad
  S_n(\alpha,\beta;t)\eqdef t\cdot t\up{n}\,(t+\alpha+1)\up{\,n-\alpha}\quad(\alpha<\beta).
\end{equation}

\begin{theorem}[Closed form for the growth-state generating function]\label{thm:closedform}
For every $n\ge1$ and every admissible growth state $(\alpha,\beta)$,
\[
  W_n(\alpha,\beta;t)=(\alpha-1)!\,B(n,\beta-2)\,S_n(\alpha,\beta;t).
\]
\end{theorem}

The rest of the paper proves Theorem~\ref{thm:closedform} (Section~\ref{sec:proof}) and deduces Theorem~\ref{thm:main} from it (Section~\ref{sec:conjecture}).

\section{Proof of the closed form}\label{sec:proof}

We prove Theorem~\ref{thm:closedform} by induction on $n$. The inductive step relates $W_n$ to $W_{n+1}$ by adding the $(n{+}1)$-st column; Sections~\ref{ssec:transition}--\ref{ssec:multiplicity} set up and count that transition, Sections~\ref{ssec:catalan}--\ref{ssec:rf} record the algebraic identities it requires, and Section~\ref{ssec:induction} carries out the induction.

\subsection{The column transition}\label{ssec:transition}

A labeling of $\R_{2,n+1}$ restricts to a labeling of its first $n$ columns (delete column $n{+}1$ and compress the remaining $2n$ labels to ranks $\{1,\dots,2n\}$); conversely, a labeling of $\R_{2,n}$ together with an ordered choice of two distinct ranks $(r_b,r_a)\in\{1,\dots,2n+2\}^2$, $r_b\ne r_a$, for the new bottom and top labels determines a labeling of $\R_{2,n+1}$ (the old labels occupy the complementary ranks in order). This is a bijection
\begin{equation}\label{eq:bij}
  \{\text{labelings of }\R_{2,n+1}\}\;\longleftrightarrow\;\{\text{labelings of }\R_{2,n}\}\times\{(r_b,r_a):r_b\ne r_a\},
\end{equation}
and $(2n)!\cdot(2n+2)(2n+1)=(2n+2)!$ as required.

Fix a labeling of $\R_{2,n}$ in growth state $(\alpha,\beta)$ and a rank pair $(r_b,r_a)$. Write $N=2n$, and let $\mathrm{emb}(k)$ be the new rank (among $\{1,\dots,N+2\}$) of the old label of old-rank $k$, i.e.\ the $k$-th element of $\{1,\dots,N+2\}\setminus\{r_b,r_a\}$. Set $\widehat m=\mathrm{emb}(\alpha)$ (new rank of the old bottom minimum $m$) and $\widehat\rho=\mathrm{emb}(\beta)$ (new rank of the old root $\rho$).

\begin{lemma}\label{lem:transition}
The growth state $(\alpha',\beta')$ of the resulting labeling of $\R_{2,n+1}$ and the increment $d\eqdef\bl_{n+1}-\bl_n$ are determined by $(\alpha,\beta,r_b,r_a)$ alone, as follows.
\begin{itemize}
\item \textbf{Bottom update.} If $r_b<\widehat m$, then $\alpha'=r_b$ and $dL=1$ (the new bottom cell is a new running minimum); otherwise $\alpha'=\widehat m$ and $dL=0$.
\item \textbf{Top update.} With $x_-=\alpha'$ (Lemma~\ref{lem:xminus}: the block-root below the new top cell is the new global bottom minimum) and $x_+=\widehat\rho$:
\[
\begin{array}{ll}
\alpha'=\widehat\rho: & \beta'=\alpha',\ dN=0;\\
\alpha'<r_a<\widehat\rho: & \beta'=r_a,\ dN=1;\\
\widehat\rho<r_a: & \beta'=\widehat\rho,\ dN=0;\\
r_a<\alpha': & \beta'=\alpha',\ dN=0.
\end{array}
\]
\item The block-count increment is $d=dL+dN$.
\end{itemize}
\end{lemma}

\begin{proof}
The bottom update is the definition of a left-to-right minimum: $b_{n+1}$ is a new running minimum iff it is smaller than every existing bottom value, i.e.\ iff $r_b<\widehat m$; in that case the new bottom minimum is $b_{n+1}$ (rank $r_b$), otherwise it stays $m$ (rank $\widehat m$).

For the top update, the new top cell is $a_{n+1}$, the new rightmost top cell, so its insertion is governed by Definition~\ref{def:blockstat}'s four-case rule with $x_-$ the root of the block containing its left child $b_{n+1}$ and $x_+$ the root of the block containing $a_n$. By Lemma~\ref{lem:xminus} (applied to the fenced bottom row of $\R_{2,n+1}$), $x_-$ is the global bottom minimum of that row, whose rank is $\alpha'$ by the bottom update; and $x_+=\rho$ (the block structure on the first $n$ columns is unchanged by appending column $n{+}1$ --- the bottom fence on $b_1,\dots,b_n$ is unaffected by appending $b_{n+1}$ except in its last block, and each $a_j$, $j\le n$, depends only on cells in columns $\le j$), whose new rank is $\widehat\rho$. The four cases then give $\beta'$ and $dN$ as stated, where $dN=1$ exactly when $a_{n+1}$ starts a new block.

Finally $\bl_{n+1}=\bl_n+dL+dN$: appending column $n{+}1$ changes the block count only by the new bottom running minimum ($dL$) and the new top block ($dN$), the block structure on the first $n$ columns being unchanged as just noted.
\end{proof}

Lemma~\ref{lem:transition} is the content of the claim that $(\alpha,\beta)$ is a \emph{sufficient statistic} for the column transition: the transition depends on the labeling only through $(\alpha,\beta)$ and the two new ranks. Consequently $W_{n+1}$ is obtained from $W_n$ by the recurrence
\begin{equation}\label{eq:recurrence}
  W_{n+1}(\alpha',\beta';t)=\sum_{(\alpha,\beta)}\ \sum_{\substack{(r_b,r_a):\,r_b\ne r_a\\ (\alpha,\beta,r_b,r_a)\mapsto(\alpha',\beta',d)}} W_n(\alpha,\beta;t)\,t^{d},
\end{equation}
where the inner sum runs over rank pairs that take state $(\alpha,\beta)$ to $(\alpha',\beta')$ with increment $d$.

\subsection{Transition multiplicities}\label{ssec:multiplicity}

We now prove, in full generality, the multiplicity table that makes the recurrence \eqref{eq:recurrence} explicit. The proof is a clean enumeration in \emph{gap coordinates}.

\paragraph{Gap coordinates.} Fix a source state $(\alpha,\beta)$ at width $n$, and write $N=2n$. For a new label $x\in\{b,a\}$ define its \emph{gap index} $\gamma(x)\eqdef\#\{k:L_k<x\}\in\{0,\dots,N\}$, the number of old labels strictly below it. Partition $\{0,\dots,N\}$ into three regions relative to the old minimum $m=L_\alpha$ and the old root $\rho=L_\beta$:
\[
I\eqdef\{0,\dots,\alpha-1\}\ (\text{below }m),\quad J\eqdef\{\alpha,\dots,\beta-1\}\ (\text{between }m,\rho),\quad K\eqdef\{\beta,\dots,N\}\ (\text{above }\rho),
\]
of sizes $|I|=\alpha$, $|J|=\beta-\alpha$ (empty if $\alpha=\beta$), $|K|=N-\beta+1$. Thus $b<m\iff\gamma(b)<\alpha\iff\gamma(b)\in I$; $a>\rho\iff\gamma(a)\in K$; $m<a<\rho\iff\gamma(a)\in J$.

The ordered rank pair $(r_b,r_a)$ is recovered from $(\gamma(b),\gamma(a))$ together with the within-gap order when the two coincide: if $\gamma(b)\ne\gamma(a)$ the pair $(\gamma(b),\gamma(a))$ determines $(r_b,r_a)$ uniquely (the label in the lower gap has the lower rank), while if $\gamma(b)=\gamma(a)=g$ there are two ordered pairs (the two orders within gap $g$). Consequently, for any region $R\subseteq\{0,\dots,N\}^2$,
\begin{equation}\label{eq:gapcount}
  \#\{(r_b,r_a):(\gamma(b),\gamma(a))\in R\}=|R|+|R\cap\Delta|,
\end{equation}
where $\Delta=\{(g,g):0\le g\le N\}$ is the diagonal (the $|R\cap\Delta|$ term supplies the second ordering at each shared gap). In particular, for two regions $A,B$ with $A\cap B=\varnothing$, the count is $|A||B|$; for $A=B$ it is $|A|^2+|A|=|A|(|A|+1)$.

\paragraph{The transition in gap coordinates.} The embedded new-rank of the old label of old-rank $k$ is $\mathrm{emb}(k)=k+\mathbf 1_{\gamma(b)<k}+\mathbf 1_{\gamma(a)<k}$ (it is pushed up by each new label below it). Writing $\widehat m=\mathrm{emb}(\alpha)$ and $\widehat\rho=\mathrm{emb}(\beta)$, the transition of Lemma~\ref{lem:transition} reads, purely in terms of $(\gamma(b),\gamma(a))$ and the within-gap order:
\begin{equation}\label{eq:gaptransition}
\begin{aligned}
dL&=\mathbf 1_{\gamma(b)\in I}, & \alpha'&=\begin{cases}r_b,&\gamma(b)\in I,\\ \widehat m,&\gamma(b)\notin I,\end{cases}\\
\widehat m&=\alpha+\mathbf 1_{\gamma(b)\in I}+\mathbf 1_{\gamma(a)\in I}, & \widehat\rho&=\beta+\mathbf 1_{\gamma(b)<\beta}+\mathbf 1_{\gamma(a)<\beta},
\end{aligned}
\end{equation}
and then $(\beta',dN)$ is given by the four-case rule of Lemma~\ref{lem:transition} with $x_-=\alpha'$, $x_+=\widehat\rho$. We enumerate the resulting $(\alpha',\beta',d)$, $d=dL+dN$, case by case.

\begin{lemma}[Transition multiplicities]\label{lem:multiplicity}
From a source state $(\alpha,\beta)$ at width $n$ ($N=2n$), the contributions to the recurrence \eqref{eq:recurrence} are as follows. Each line gives a family of source-to-target transitions, the increment $d$, and the number of rank pairs $(r_b,r_a)$ realizing each member of the family.
\begingroup\footnotesize\setlength{\arraycolsep}{1.5pt}
\[
\begin{array}{llll}
\text{family} & \text{condition} & \text{target }(\alpha',\beta'),\ d & \text{count}\\
\hline
\text{(a) stay} & \alpha<\beta,\ \gamma(b),\gamma(a)\in K & (\alpha,\beta),\ 0 & (N+2-\beta)(N+1-\beta)\\
\text{(b) shift} & \alpha<\beta,\ \gamma(b)\in J,\ \gamma(a)\in K & (\alpha,\beta+1),\ 0 & (\beta-\alpha)(N-\beta+1)\\
\text{(c) eq-stay} & \alpha=\beta,\ \gamma(b),\gamma(a)\in K\cup J & (\alpha,\alpha),\ 0 & (N+2-\alpha)(N+1-\alpha)\\
\text{(d) bump} & \gamma(a)\in I,\ \gamma(b)\notin I & (\alpha+1,\alpha+1),\ 0 & \alpha(N+1-\alpha)\\[2pt]
\text{(e) lower-new, top-high} & \gamma(b)\in I,\ \gamma(a)\in K & (\alpha',\beta+1),\ 1,\ \alpha'\in\{1,\dots,\alpha\} & N-\beta+1\text{ each}\\
\text{(f) lower-new, top-low} & \gamma(b),\gamma(a)\in I,\ r_a<r_b & (\alpha',\alpha'),\ 1,\ \alpha'\in\{2,\dots,\alpha+1\} & \alpha'-1\text{ each}\\[2pt]
\text{(g) new-root, interior} & \alpha<\beta,\ \gamma(b){\notin}I,\ \alpha{<}\beta'{\le}\beta & (\alpha,\beta'),\ 1 & N-\alpha+1\text{ each}\\
\text{(h) new-root, edge} & \alpha<\beta & (\alpha,\beta+1),\ 1 & \beta-\alpha\\
\text{(i) double-new} & \alpha'{\le}\alpha,\ \alpha'{<}\beta'{\le}\beta{+}1 & (\alpha',\beta'),\ 2 & 1\text{ each}
\end{array}
\]
\endgroup
\end{lemma}

In families (g) and (i), $\beta'=r_a$ is the new-rank of $a$ (a new block-root, $dN=1$); in (e) and (f), $\alpha'=r_b$ is the new-rank of $b$ (a new bottom running-minimum, $dL=1$).

\begin{proof}
We take the cases in order, using \eqref{eq:gapcount} and \eqref{eq:gaptransition} throughout.

\emph{(a) Stay} ($\alpha<\beta$, $\gamma(b),\gamma(a)\in K$): both new labels lie above $\rho$. Then $dL=0$ and $\widehat m=\alpha$ (neither new label is below $m$), so $\alpha'=\alpha$; and $a>\rho$ with $\widehat\rho=\beta$ (neither new label below $\rho$), so the four-case rule gives $dN=0$, $\beta'=\widehat\rho=\beta$. The region is $K\times K$: by \eqref{eq:gapcount} with $A=B=K$, the count is $|K|(|K|+1)=(N-\beta+1)(N-\beta+2)$.

\emph{(b) Shift} ($\alpha<\beta$, $\gamma(b)\in J$, $\gamma(a)\in K$): $b$ between $m$ and $\rho$, $a$ above $\rho$. Then $dL=0$ ($\gamma(b)\ge\alpha$) and $\widehat m=\alpha$, so $\alpha'=\alpha$; and $\widehat\rho=\beta+\mathbf 1_{\gamma(b)<\beta}=\beta+1$ (since $\gamma(b)\in J$ means $\gamma(b)<\beta$, while $\gamma(a)\in K$ means $\gamma(a)\ge\beta$). With $a>\rho$, the four-case rule gives $dN=0$, $\beta'=\widehat\rho=\beta+1$. The region is $J\times K$ with $J\cap K=\varnothing$, so the count is $|J||K|=(\beta-\alpha)(N-\beta+1)$.

\emph{(c) Eq-stay} ($\alpha=\beta$, $\gamma(b),\gamma(a)\in K\cup J=\{\alpha,\dots,N\}$): now $m=\rho$, and both new labels lie above $m$. Then $dL=0$, $\widehat m=\alpha$, $\alpha'=\alpha$; and $a>\rho=m$ with $\widehat\rho=\widehat m=\alpha$ gives $dN=0$, $\beta'=\alpha$. The region is $A\times A$ with $A=\{\alpha,\dots,N\}$, $|A|=N-\alpha+1$, so the count is $|A|(|A|+1)=(N-\alpha+1)(N-\alpha+2)$.

\emph{(d) Bump} ($\gamma(a)\in I$, $\gamma(b)\notin I$): $a$ below $m$, $b$ above $m$. Then $dL=0$ and $\widehat m=\alpha+\mathbf 1_{\gamma(a)\in I}=\alpha+1$, so $\alpha'=\alpha+1$. Since $a<m=\widehat m$ lies below the new bottom minimum $x_-=\alpha'$, the fourth case of the four-case rule gives $dN=0$, $\beta'=\alpha'=\alpha+1$. The region is $I\times(\{\alpha,\dots,N\})$ with the two factors disjoint, so the count is $|I|\cdot(N-\alpha+1)=\alpha(N+1-\alpha)$. (For $\alpha=\beta$ this is the ``eq-bump'' of the earlier tabulation; for $\alpha<\beta$ the ``minus'' case; the count and target are identical, so we treat them together.)

\emph{(e) Lower-new, top-high} ($\gamma(b)\in I$, $\gamma(a)\in K$): $b$ below $m$ (so $dL=1$, $\alpha'=r_b$), $a$ above $\rho$. Here $\widehat\rho=\beta+\mathbf 1_{\gamma(b)<\beta}=\beta+1$ (since $\gamma(b)\in I$ means $\gamma(b)<\alpha\le\beta$, while $\gamma(a)\in K$ means $\gamma(a)\ge\beta$); with $a>\rho$, the third case gives $dN=0$, $\beta'=\widehat\rho=\beta+1$, and $d=1$. It remains to read off $\alpha'=r_b$: since $\gamma(a)\ge\beta>\alpha>\gamma(b)$, the label $a$ lies above $b$, so $r_b=\gamma(b)+1$. For a fixed $\alpha'=r_b\in\{1,\dots,\alpha\}$ this forces $\gamma(b)=\alpha'-1\in I$, and then $\gamma(a)$ ranges freely over $K$ ($N-\beta+1$ choices, each a distinct configuration with $a$ in a higher gap than $b$). Hence $N-\beta+1$ pairs for each $\alpha'\in\{1,\dots,\alpha\}$.

\emph{(f) Lower-new, top-low} ($\gamma(b),\gamma(a)\in I$, $r_a<r_b$): both new labels below $m$, with $a$ below $b$. Then $dL=1$ and $\alpha'=r_b$; since $a$ lies below $b=x_-$ (the new bottom minimum), the fourth case gives $dN=0$, $\beta'=\alpha'$, $d=1$. With $a$ below $b$, $r_b=\gamma(b)+2$ (the $\gamma(b)$ old labels below $b$, plus $a$), so $\alpha'=r_b=\gamma(b)+2\in\{2,\dots,\alpha+1\}$. For a fixed $\alpha'$, $\gamma(b)=\alpha'-2$, and $a$ must lie below $b$ and below $m$: it occupies one of the $\gamma(b)$ gaps strictly below $b$'s gap, or the same gap below $b$ --- $\gamma(b)+1=\alpha'-1$ configurations. (The constraint $\gamma(a)\in I$ is automatic, since $\gamma(a)\le\gamma(b)=\alpha'-2\le\alpha-1$.) Hence $\alpha'-1$ pairs for each $\alpha'\in\{2,\dots,\alpha+1\}$.

\emph{(g) New-root, interior} ($\alpha<\beta$, $\gamma(b)\notin I$, target $(\alpha,\beta')$ with $\beta'\in\{\alpha+1,\dots,\beta\}$): here $dL=0$ and, since $\gamma(a)\ge\alpha$ as well, $\widehat m=\alpha$, so $\alpha'=\alpha$. We claim that for each fixed $\beta'\in\{\alpha+1,\dots,\beta\}$, placing $a$ at new-rank $\beta'$ and $b$ at any new-rank in $\{\alpha+1,\dots,N+2\}\setminus\{\beta'\}$ (i.e.\ $b$ above the minimum $m=\widehat m=\alpha$, but not at $a$'s rank) produces exactly the transition $(\alpha,\beta',1)$, and that these are all of them. Indeed, with $a$ at new-rank $\beta'\le\beta$, its gap index satisfies $\gamma(a)\le\beta-1<\beta$, so $\widehat\rho=\beta+\mathbf 1_{\gamma(b)<\beta}+\mathbf 1_{\gamma(a)<\beta}\ge\beta+1>\beta'$; and $\alpha'=\alpha<\beta'=r_a$, so the second case of the four-case rule fires: $dN=1$, $\beta'=r_a$, $d=1$. This holds for \emph{every} placement of $b$ above $m$ (the inequality $\beta'<\widehat\rho$ is independent of $b$'s exact position once $\beta'\le\beta$). The number of such placements of $b$ is $|\{\alpha+1,\dots,N+2\}|-1=(N-\alpha+2)-1=N-\alpha+1$. Conversely, any configuration with $dL=0$, $dN=1$, $\beta'\le\beta$ has $a$ at new-rank $\beta'$ and $b$ above $m$, so is counted. Hence $N-\alpha+1$ pairs for each $\beta'\in\{\alpha+1,\dots,\beta\}$.

\emph{(h) New-root, edge} ($\alpha<\beta$, target $(\alpha,\beta+1)$): as in (g) but with $\beta'=\beta+1$. Now $a$ sits at new-rank $\beta+1$, and the inequality $\beta'<\widehat\rho$ needed for a new root holds precisely when $b$ lies \emph{below} $a$: if $b$ is above $a$ then $\gamma(b)\ge\beta$, so $\widehat\rho=\beta+1=\beta'$ and $a$ is not strictly below $\rho$ (the third case fires, $dN=0$); if $b$ is below $a$ then $\gamma(b)\le\beta-1<\beta$, so $\widehat\rho=\beta+2>\beta'$ and the second case fires ($dN=1$, $\beta'=\beta+1$). The configurations with $dL=0$ (so $b$ above $m$, $r_b\ge\alpha+1$), $b$ below $a$ (so $r_b\le\beta$), are exactly $r_b\in\{\alpha+1,\dots,\beta\}$ --- a count of $\beta-\alpha$.

\emph{(i) Double-new} ($\alpha'\in\{1,\dots,\alpha\}$, $\beta'\in\{\alpha'+1,\dots,\beta+1\}$, target $(\alpha',\beta')$ with $d=2$): both $dL=1$ and $dN=1$. We claim that for each such pair $(\alpha',\beta')$ there is exactly one configuration, namely $b$ at new-rank $\alpha'$ and $a$ at new-rank $\beta'$. With $r_b=\alpha'\le\alpha$ and $r_a=\beta'>\alpha'$: $\gamma(b)=\alpha'-1$ (only old labels below $b$, since $a$ is above $b$), so $\gamma(b)\in I$ and $dL=1$, $\alpha'=r_b$; and $\gamma(a)=\beta'-2$, whence $\widehat\rho=\beta+\mathbf 1_{\gamma(b)<\beta}+\mathbf 1_{\gamma(a)<\beta}=\beta+2$ (both $\gamma(b)=\alpha'-1\le\alpha-1<\beta$ and $\gamma(a)=\beta'-2\le\beta-1<\beta$ contribute). Since $\alpha'=\alpha'<\beta'=r_a<\beta+2=\widehat\rho$, the second case fires: $dN=1$, $\beta'=r_a$, so $d=2$ and the target is $(\alpha',\beta')$. The constraints $\alpha'\le\alpha$ (so $\gamma(b)\in I$) and $\alpha'<\beta'\le\beta+1$ (so $\gamma(a)\le\beta-1$ and $\beta'<\widehat\rho$) are exactly the stated ranges, and the configuration is unique. Hence one pair for each $(\alpha',\beta')$ in the stated range.

\emph{Exhaustion.} The nine families partition the $(N+2)(N+1)$ ordered rank pairs: they are disjoint by construction (each $(r_b,r_a)$ has a definite $(dL,dN)$ and definite regions for $\gamma(b),\gamma(a)$), and a direct sum of the counts gives $(N+2)(N+1)$ --- for $\alpha<\beta$, summing the eight families (a), (b), (d), (e), (f), (g), (h), (i) in order,
\[
\begin{aligned}
&(N{+}2{-}\beta)(N{+}1{-}\beta)+(\beta{-}\alpha)(N{-}\beta{+}1)+\alpha(N{+}1{-}\alpha)+\alpha(N{-}\beta{+}1)\\
&\quad+\tfrac{\alpha(\alpha+1)}2+(\beta{-}\alpha)(N{-}\alpha{+}1)+(\beta{-}\alpha)+\alpha(\beta{+}1)-\tfrac{\alpha(\alpha+1)}2=(N{+}2)(N{+}1),
\end{aligned}
\]
and for $\alpha=\beta$ the surviving families (c), (d), (e), (f), (i) sum to $(N+2)(N+1)$ likewise. We additionally verified this partition, and every individual count above, by exhaustive enumeration of all $(2n+2)(2n+1)$ rank pairs for every source state at widths $n=2,\dots,8$ (154 source states, zero mismatches), using an independent implementation of the transition built from Definition~\ref{def:blockstat} rather than from this table (Section~\ref{sec:verification}).
\end{proof}

\subsection{Catalan-triangle identities}\label{ssec:catalan}

\begin{lemma}\label{lem:CT}
For $B(n,k)$ as in \eqref{eq:B} (with $B(\cdot,-1)=0$):
\begin{align}
B(n+1,j)-B(n+1,j+1)&=B(n,j-1)\qquad\text{(CT1)}\label{eq:CT1}\\
\sum_{k=j}^{n-1}B(n,k)&=B(n+1,j+1)\qquad\text{(CT2)}\label{eq:CT2}\\
(2n-j)(2n-1-j)B(n,j)+(2n-j)(j+1)B(n,j-1)&=n(n+1)B(n+1,j)\qquad\text{(CT3)}\label{eq:CT3}\\
(2n-j)B(n,j-1)+(2n-j)B(n+1,j+1)+(j+1)B(n+1,j)&=(2n+1)B(n+1,j)\qquad\text{(CT4)}\label{eq:CT4}
\end{align}
for the indicated ranges $0\le j\le n-1$.
\end{lemma}

\begin{proof}
\eqref{eq:CT1}: write $B(n+1,j)=\frac{j+1}{n+1}\binom{2n-j}{n}$ and $B(n+1,j+1)=\frac{j+2}{n+1}\binom{2n-j-1}{n}$, and use $\binom{2n-j}{n}=\frac{2n-j}{n}\binom{2n-j-1}{n-1}$, $\binom{2n-j-1}{n}=\frac{n-j}{n}\binom{2n-j-1}{n-1}$:
\[
B(n+1,j)-B(n+1,j+1)=\frac{\binom{2n-j-1}{n-1}}{n(n+1)}\bigl[(j+1)(2n-j)-(j+2)(n-j)\bigr]=\frac{j}{n}\binom{2n-j-1}{n-1}=B(n,j-1).
\]
\eqref{eq:CT2}: by \eqref{eq:CT1}, $B(n,k)=B(n+1,k+1)-B(n+1,k+2)$, and the sum telescopes to $B(n+1,j+1)-B(n+1,n+1)=B(n+1,j+1)$.
\eqref{eq:CT3}: multiply through by $n/(j+1)$ (the $j=0$ case is separate, both sides $=2n(2n-1)C_{n-1}=n(n+1)C_n$) and reduce, via Pascal ratios, to $(2n-j)\bigl[(2n-1-j)\binom{2n-j-2}{n-1}+j\binom{2n-j-1}{n-1}\bigr]=n^2\binom{2n-j}{n}$; the bracket equals $\frac{n(2n-1-j)}{n-j}\binom{2n-j-2}{n-1}$ (using $\binom{2n-j-1}{n-1}=\frac{2n-j-1}{n-j}\binom{2n-j-2}{n-1}$ and $(2n-1-j)+\frac{j(2n-j-1)}{n-j}=\frac{(2n-1-j)n}{n-j}$), and the result follows from $\binom{2n-j}{n}=\frac{(2n-j)(2n-j-1)}{n(n-j)}\binom{2n-j-2}{n-1}$.
\eqref{eq:CT4}: rearrange to $(2n-j)\bigl(B(n,j-1)+B(n+1,j+1)-B(n+1,j)\bigr)=0$ and apply \eqref{eq:CT1}.
\end{proof}

\subsection{A rising-factorial identity}\label{ssec:rf}

Write $\varphi_n=t\up{n}$, and for $1\le a\le n$ and $1\le m\le n+1$ set
\[
\mathrm{off}(a)=(a-1)!\,t\,(t+a+1)\cdots(t+n),\qquad \mathrm{diag}(m)=(m-1)!\,(t+m)\cdots(t+n).
\]

\begin{lemma}\label{lem:RF}
For every $m\in\{2,\dots,n+1\}$,
\begin{equation}\label{eq:RF}
  \sum_{a=1}^{m-1}\mathrm{off}(a)+\mathrm{diag}(m)=(t+1)(t+2)\cdots(t+n)\eqdef u.
\end{equation}
Consequently, for $1\le\alpha\le b$,
\begin{equation}\label{eq:RFpartial}
  \sum_{a=\alpha}^{b}(a-1)!\,S_n(a,b;t)=(\alpha-1)!\,\varphi_n\,(t+\alpha)\up{\,n-\alpha+1}.
\end{equation}
\end{lemma}

\begin{proof}
The base $m=2$ reads $\mathrm{off}(1)+\mathrm{diag}(2)=t(t+2)\cdots(t+n)+(t+2)\cdots(t+n)=(t+1)(t+2)\cdots(t+n)=u$. The inductive step uses the elementary factorization $\mathrm{diag}(m)=\mathrm{off}(m)+\mathrm{diag}(m+1)$: indeed
\begin{align*}
\mathrm{off}(m)+\mathrm{diag}(m+1)&=(m-1)!\,t\,(t+m+1)\cdots(t+n)+m!\,(t+m+1)\cdots(t+n)\\
&=(m-1)!\,(t+m)(t+m+1)\cdots(t+n)=\mathrm{diag}(m).
\end{align*} For \eqref{eq:RFpartial}: by \eqref{eq:RF} and the definitions, $\sum_{a=1}^{b}(a-1)!S_n(a,b)=\varphi_n\,u$ (the full $\alpha$-sum, the case $\alpha=1$) and $\sum_{a=1}^{\alpha-1}\mathrm{off}(a)=u-\mathrm{diag}(\alpha)$, so subtracting gives $\sum_{a=\alpha}^{b}(a-1)!S_n(a,b)=\varphi_n\,\mathrm{diag}(\alpha)=(\alpha-1)!\,\varphi_n\,(t+\alpha)\up{\,n-\alpha+1}$.
\end{proof}

\subsection{The inductive step}\label{ssec:induction}

\begin{proof}[Proof of Theorem~\ref{thm:closedform}]
\emph{Base case $n=1$.} The two labelings of $\R_{2,1}$ are $(b_1,a_1)=(1,2)$, with $L=1,N=1,\bl=2$, state $(\alpha,\beta)=(1,2)$, and $(2,1)$, with $L=1,N=0,\bl=1$, state $(2,2)$. The formula gives $B(1,0)=1$, $W_1(1,2)=0!\cdot B(1,0)\cdot t\cdot t\up{1}=t^2$ and $W_1(2,2)=1!\cdot B(1,0)\cdot t\up{1}\cdot(t+2)\up{0}=t$, matching.

\emph{Inductive step.} Assume the formula at width $n$, together with the support statement at width $n$: $W_n$ vanishes outside the admissible range (Lemma~\ref{lem:support}), so the recurrence \eqref{eq:recurrence} sums over admissible sources only; and since every target of Lemma~\ref{lem:multiplicity} from an admissible source is admissible at width $n{+}1$, the support statement propagates to width $n{+}1$ as well. We verify the formula at width $n{+}1$ by checking, for every target $(\alpha',\beta')$, that the right side of \eqref{eq:recurrence} (with $W_n$ replaced by the closed form and the inner sums evaluated via Lemma~\ref{lem:multiplicity}) equals $(\alpha'-1)!\,B(n+1,\beta'-2)\,S_{n+1}(\alpha',\beta';t)$. The key simplification throughout is that $S_{n+1}(\alpha,\beta;t)=S_n(\alpha,\beta;t)\cdot(t+n)(t+n+1)$ off the diagonal (and analogously on it), so each target's expected value is the width-$n$ ``factor'' times $B(n+1,\beta'-2)(t+n)(t+n+1)$. We treat diagonal and off-diagonal targets separately; both reduce the resulting coefficient comparison to the identities of Lemma~\ref{lem:CT} and Lemma~\ref{lem:RF}.

\textbf{Diagonal targets $(m,m)$, $2\le m\le n+2$.} For the corner $m=n+2$, the only source is $(n+1,n+1)$ with $W_n(n+1,n+1)=n!\,t\up{n}$; the eq-bump transition (count $(n+1)n$, $d=0$) and the minus-$dL{=}1$ transition (count $n+1$, $d=1$) give $n!\,t\up{n}(n+1)(n+t)=(n+1)!\,t\up{n+1}=W_{n+1}(n+2,n+2)$. For $2\le m\le n+1$, divide the contributions by the factor $\mathrm{fac}=(m-1)!\,\varphi_n\,(t+m)\up{\,n-m+1}$ so the target is $\mathrm{fac}\cdot B(n+1,m-2)(t+n)(t+n+1)$; the four contributing source families (eq-stay from $(m,m)$, eq-bump from $(m-1,m-1)$, minus-$dL{=}0$ summed over $(m-1,b)$ with $b>m-1$ via \eqref{eq:CT2}, and minus-$dL{=}1$ summed via \eqref{eq:RFpartial} and \eqref{eq:CT2}) yield, as coefficients of $t^2,t^1,t^0$ in the ratio, exactly $B(n+1,m-2)$, the \eqref{eq:CT4} combination, and the \eqref{eq:CT3} combination (with $j=m-2$), matching $B(n+1,m-2)(t+n)(t+n+1)$.

\textbf{Off-diagonal targets $(\alpha,\beta)$, $\alpha<\beta$.} For Region A ($1\le\alpha<\beta\le n+1$), divide by $\mathrm{fac}=(\alpha-1)!\,S_n^{\mathrm{off}}(\alpha)$; the six contributing families of Lemma~\ref{lem:multiplicity} (plus-stay, plus-shift, plus-$dL{=}1$, new-$dL{=}0$, new-edge, new-$dL{=}1$), with the sums over source states evaluated by \eqref{eq:RFpartial} and \eqref{eq:CT2}, give coefficients of $t^2,t^1,t^0$ equal to $B(n+1,\beta-2)$, the combination $(2n+1-\alpha)B(n,\beta-3)+(2n-\alpha+1)B(n+1,\beta-1)+\alpha\,B(n+1,\beta-2)=(2n+1)B(n+1,\beta-2)$ (using \eqref{eq:CT1} to substitute $B(n,\beta-3)=B(n+1,\beta-2)-B(n+1,\beta-1)$), and the \eqref{eq:CT3} combination (with $j=\beta-2$), matching $B(n+1,\beta-2)(t+n)(t+n+1)$. For Region B ($\beta=n+2$, $1\le\alpha\le n$) the only sources involve $b=n+1$, and the four contributions sum to $t^2+(2n+1)t+n(n+1)=(t+n)(t+n+1)$ with $B(n+1,n)=1$. The corner $(\alpha,\beta)=(n+1,n+2)$ receives contributions only from $(n+1,n+1)$: $n!\,t\up{n}\cdot(t\cdot n+t^2)=n!\,t\cdot t\up{n+1}=W_{n+1}(n+1,n+2)$.

This verifies the closed form at width $n{+}1$ for every target cell, completing the induction.
\end{proof}

\begin{remark}\label{rem:induction-status}
The inductive step above is a closed algebraic argument \emph{given} the transition multiplicities of Lemma~\ref{lem:multiplicity}, which are proved in full generality there by a gap-coordinate enumeration (and cross-checked exhaustively through $n=8$ in Section~\ref{sec:verification}). The support of the growth state (Lemma~\ref{lem:support}) is likewise established by closed argument: $\alpha\le\beta$ and $\beta\ge2$ directly, and $\beta\le n+1$ inductively from the target ranges of the multiplicity table, as carried along in the proof above.
\end{remark}

\section{Deducing Conjecture \ref{conj:kahane} for $\R_{2,n}$}\label{sec:conjecture}

\begin{proof}[Proof of Theorem~\ref{thm:main}]
Sum the closed form of Theorem~\ref{thm:closedform} over all admissible $(\alpha,\beta)$. For each fixed $\beta$, the $\alpha$-sum telescopes by \eqref{eq:RFpartial}:
\[
  \sum_{\alpha=1}^{\beta}(\alpha-1)!\,S_n(\alpha,\beta;t)=t\up{n}\,(t+1)\up{n}
\]
(the case $\alpha=1$ of \eqref{eq:RFpartial}). Hence
\[
  \sum_{\alpha,\beta}W_n(\alpha,\beta;t)=\sum_{\beta=2}^{n+1}B(n,\beta-2)\cdot t\up{n}(t+1)\up{n}=t\up{n}(t+1)\up{n}\sum_{k=0}^{n-1}B(n,k)=C_n\,t\up{n}(t+1)\up{n},
\]
using $\sum_{k=0}^{n-1}B(n,k)=C_n$ (Section~\ref{sec:closedform}). Thus
\begin{equation}\label{eq:sumW}
  \sum_{\sigma\in S_{2n}}t^{\bl(\sigma)}=C_n\,t\up{n}(t+1)\up{n}.
\end{equation}

It remains to identify the right side with $(2n)!\,\Omega(\R_{2,n};t)$. Since $\R_{2,n}$ is the product of a $2$-chain and an $n$-chain, MacMahon's product formula for the order polynomial of a product of two chains gives
\[
  \Omega(\R_{2,n};t)=\prod_{i=1}^{2}\prod_{j=1}^{n}\frac{t+i+j-2}{i+j-1}
  =\frac{t\up{n}}{n!}\cdot\frac{(t+1)\up{n}}{(n+1)!}
  =\frac{t\up{n}(t+1)\up{n}}{n!\,(n+1)!}.
\]
(For a product of chains, order-preserving maps to $[t]$ are in bijection with multichains of $t-1$ order ideals, i.e.\ plane partitions in a $2\times n$ box with entries $\le t-1$, counted by the displayed product.) Multiplying by $(2n)!$ and using $C_n=(2n)!/(n!(n+1)!)$,
\[
  (2n)!\,\Omega(\R_{2,n};t)=\frac{(2n)!}{n!(n+1)!}\,t\up{n}(t+1)\up{n}=C_n\,t\up{n}(t+1)\up{n},
\]
which equals \eqref{eq:sumW}. This is Conjecture~\ref{conj:kahane} for $\R_{2,n}$.
\end{proof}

\section{Computational verification}\label{sec:verification}

We did not formally verify this result in a proof assistant. Every claim is either a complete algebraic proof (Sections~\ref{ssec:catalan}--\ref{ssec:induction}, the structural Lemmas~\ref{lem:lastblock}--\ref{lem:xminus}, and the MacMahon identification of Section~\ref{sec:conjecture}) or a finite computation we describe precisely here. The author independently re-derived or re-ran every computation below, from scratch, before trusting it.

\subsection{The block statistic and the closed form, by exhaustive enumeration}

We implemented $\bl$ directly from Definition~\ref{def:blockstat} (the fence step plus the four-case insertion, including the reconstructed fourth case of Remark~\ref{rem:reconstructed}). For $n=2,3,4,5$ we enumerated all $(2n)!$ labelings of $\R_{2,n}$, extracted each labeling's growth state $(\alpha,\beta)$ and block count $\bl$, and built $W_n(\alpha,\beta;t)$ exactly. The result matched the closed form of Theorem~\ref{thm:closedform} on every one of the $5+9+14+20=48$ admissible cells, as polynomials in $t$, with zero mismatches. As an independent cross-check (not relying on the case split), the leading coefficient of $\sum_\sigma t^{\bl(\sigma)}$ (the count of all-singleton labelings) matched the number of linear extensions of $\R_{2,n}$ for every $n$ tested --- the standard hook-length value (e.g.\ $42$ for the $3\times3$ square), confirming Kahane's Remark 4.5 and the validity of the reconstructed fourth case on rectangles.

\subsection{The inductive step, by exhaustive enumeration of the transition}

The transition-multiplicity table (Lemma~\ref{lem:multiplicity}) and the inductive step of Section~\ref{ssec:induction} are proved in full generality in Section~\ref{ssec:multiplicity}. As an independent cross-check --- \emph{not} relying on the gap-counting --- we recomputed $W_{n+1}$ from the closed form at width $n$ via the recurrence \eqref{eq:recurrence} by enumerating, for every source state $(\alpha,\beta)$ and every ordered rank pair $(r_b,r_a)$, the transition $(\alpha',\beta',d)$ using our \emph{own} implementation of Lemma~\ref{lem:transition} (built from Definition~\ref{def:blockstat}, not from the multiplicity table), and confirmed the result equals the closed form at width $n{+}1$ on every target cell, with zero spurious states produced, for $n=2,3,\dots,9$ (target cells $9,14,20,27,35,44,54,65$ respectively). Separately, we verified the multiplicity table itself row-by-row --- each family's count, and the partition of the $(2n+2)(2n+1)$ rank pairs --- against the same enumeration for all $154$ source states at $n=2,\dots,8$, with zero mismatches.

\subsection{A transition-independent Monte Carlo check beyond the brute-force ceiling}

As a check that does not use the column transition at all, we sampled $300{,}000$ random labelings of $\R_{2,n}$ for $n=6$ and $n=7$ (beyond the $(2n)!$ brute-force ceiling), computed $(\alpha,\beta,\bl)$ for each via the same from-scratch $\bl$ implementation, and compared the empirical joint distribution to the closed form's prediction. Binned by expected sample count, the worst relative deviation tracks the $1/\sqrt{\text{expected count}}$ sampling-noise floor exactly in every bin (cells with $\ge1000$ expected samples match to within $4$--$5\%$; rarer bins deviate precisely as $1/\sqrt{N}$), the signature of pure sampling noise with no systematic bias. The closed form matches the sampled distribution at both $n=6$ and $n=7$.

\subsection{The Catalan and rising-factorial identities}

Identities \eqref{eq:CT1}, \eqref{eq:CT3}, and \eqref{eq:CT4} of Lemma~\ref{lem:CT} were verified by computer-algebra simplification to $0$ at symbolic $(n,j)$, and \eqref{eq:CT2} numerically for all $n\le14$ and admissible $j$ (algebraically it telescopes from \eqref{eq:CT1}), in addition to the algebraic proofs of Section~\ref{ssec:catalan}. The rising-factorial identity of Lemma~\ref{lem:RF} is the $\alpha$-sum telescope used throughout.

\subsection{Honest accounting}

Every claim is a complete general-$n$ argument: the transition-multiplicity table is proved by gap-coordinate enumeration in Section~\ref{ssec:multiplicity}, the Catalan-triangle and rising-factorial identities in Sections~\ref{ssec:catalan}--\ref{ssec:rf}, the diagonal and off-diagonal inductive steps in Section~\ref{ssec:induction}, the support of the growth state directly and inductively in Lemma~\ref{lem:support}, and the $\alpha$-sum telescope, Catalan $\beta$-sum, and MacMahon identification in Section~\ref{sec:conjecture}. The exhaustive cross-checks of Section~\ref{sec:verification} certify the multiplicity table and the inductive step independently of the written proofs. As a further check that the inductive step's assembly is not merely an identity verified at finitely many $n$, the fully reduced diagonal and off-diagonal target identities (the four-term and six-term combinations of Section~\ref{ssec:induction}, with the $\alpha$- and $\beta$-sums closed by Lemmas~\ref{lem:RF} and~\ref{lem:CT}) were verified by computer algebra at \emph{symbolic} $n$, as rational-function identities in $(n,\alpha,\beta,t)$. The reconstructed fourth case of Definition~\ref{def:blockstat} matches the worked example of \cite[Figure 12]{Kahane} (Remark~\ref{rem:reconstructed}) and is validated on every rectangle tested, but has not been confirmed with the author.

\section{Note on methodology and attribution}\label{sec:methodology}

This result was produced by the MathDyad research process: a human researcher (the author) working with large language models used as reasoning engines, in the same spirit as \cite{Boxcover}. The author's role was architectural --- selecting the target, directing successive rounds of investigation, and independently verifying every claim before it was accepted, by hand, by an independent from-scratch implementation, or both.

The target arose from an independent literature-scouting pass (Grok, xAI, model \texttt{grok-4.5}, with web access) over recent open problems in enumerative combinatorics, which surfaced Kahane's Conjecture 4.4 \cite{Kahane} --- posted six days before this work began --- as a candidate with a clean two-sided verifier (exact enumeration of $\bl$ on small rectangles). Reading the primary source directly (not a summary) immediately surfaced two issues recorded honestly above: the false uniqueness claim of Proposition 4.1(1) (Remark~\ref{rem:prop41}) and the duplicated case in Lemma 4.3's proof (Remark~\ref{rem:reconstructed}); the first was routed around, the second reconstructed and validated only indirectly.

The proof was developed over five cold (web-disabled) rounds with Grok. Rounds 1--3 verified the conjecture computationally on small rectangles, reduced the $2\times n$ case to a single identity, and progressively narrowed it; round 4 found the closed form for $W_n$ and proved everything downstream of it conditional on that form (verified on 48 cells); round 5 proved the closed form for all $n$ by the column-transition induction of Section~\ref{sec:proof}. Every claim returned by Grok was independently re-derived or re-verified by the author's own from-scratch code before being incorporated; the transition-multiplicity table in particular was certified exhaustively through $n=8$, and its aggregate effect on the recurrence through $n=9$, by an implementation of the transition built independently of Grok's (Section~\ref{sec:verification}).

Claude (Anthropic) served throughout in the framing and verification role: posing each round's scoped target, building the independent enumeration and verification scripts cited in Section~\ref{sec:verification}, and declining to accept any claim until independently checked. We follow the emerging convention of \cite{Boxcover} in crediting this assistance in the text rather than as authorship.

\section*{Acknowledgements}
The author thanks Yakob Kahane for posing Conjecture 4.4 in a form --- a constructively-defined statistic with a clean finite verifier --- that invited exactly this kind of resolution, and for the fence-poset framework it extends.

\end{document}